\documentclass[11pt]{amsart}
\usepackage[utf8]{inputenc}
\usepackage{amsmath,amssymb,mathtools,bm,comment}

\newtheorem{thm}{Theorem}[section]
\newtheorem{cor}[thm]{Corollary}
\newtheorem{lemma}[thm]{Lemma}

\newtheorem*{theoremA*}{Theorem A}
\theoremstyle{definition}
\newtheorem{defn}[thm]{Definition}
\theoremstyle{remark}
\newtheorem{rem}[thm]{Remark}
\numberwithin{equation}{section}

\def\bN{\mathbf{N}}
\def\bR{\mathbb{R}}

\def\og{\otimes_g}
\def\oa{\otimes_a}
\def\bd{\mathbf{d}}

\def\<{\prec}
\def\>{\succ}
\def\O{{\mathcal O}}
\def\bes{\begin{equation*}}
\def\be{\begin{equation}}
\def\ee{\end{equation}}
\def\ees{\end{equation*}}

\def\a{\alpha}
\def\b{\beta}
\def\c{\gamma}
\allowdisplaybreaks

\def\ba{{\bf a}}
\def\bb{{\bf b}}
\def\bc{{\bf c}}

\def\bd{{\bf d}}

\begin{document}

\title{Asymptotic Analysis of Nonsymmetric Gaussian and Archimedean Compound Means}
\author{Tomislav Buri\'c}
\author{Lenka Mihokovi\'c}
\author{Toni Milas}

\address{Tomislav Buri\'c\newline
\indent University of Zagreb Faculty of Electrical Engineering and Computing\newline
\indent Unska 3, 10000 Zagreb, Croatia}
\email{tomislav.buric@fer.hr}

\address{Lenka Mihokovi\'c\newline
\indent University of Zagreb Faculty of Electrical Engineering and Computing\newline
\indent Unska 3, 10000 Zagreb, Croatia\newline
\indent ORCID: {0000-0002-9561-6021}}
\email{lenka.mihokovic@fer.hr}

\address{Toni Milas\newline
\indent University of Zagreb Faculty of Mechanical Engineering and Naval Architecture\newline
\indent Ivana Lu\v{c}i\'ca 5, 10000 Zagreb, Croatia}
\email{toni.milas@fsb.unizg.hr}

\date{\today}


\begin{abstract}
This paper studies asymptotic expansions of the Gaussian and Archimedean compounds of two arbitrary bivariate means that need not be symmetric. Recursive algorithms for the coefficients of such expansions were
previously obtained only in the symmetric case. Here, this restriction is removed and recursions are derived for the general nonsymmetric setting. All coefficient calculations are carried out in the algebra of formal power
series. This introduces an additional case in the recursion, as well as singular configurations in which the formal invariance equation does not determine the coefficients recursively. The formal problem is kept separate from the existence and convergence of the iterative procedures. We also connect the first formal coefficients with Farhi's metric, recall its known global convergence criterion, and derive the corresponding local rates. The algorithms are
illustrated using weighted power means and several known identities. Finally, the method is extended to expansions with sign-dependent coefficients and is applied to the neo-Pythagorean means and to a recently introduced two-parameter family.
\end{abstract}

\maketitle
\markboth{T. Buri\'c, L. Mihokovi\'c and T. Milas}
{Nonsymmetric Gaussian and Archimedean Compound Means}

\section{Introduction and motivation}

Consider a \emph{bivariate mean} $M$, i.e.\ a
function $M\colon \bR^+ \times \bR^+\to\bR^+$
such that
\be\label{def-mean-minmax}
	 \min(s,t) \le M(s,t) \le \max(s,t).
\ee
We define the Gauss iterative algorithm for arbitrary bivariate means $M$ and $N$ as follows:
\bes
	\begin{aligned}
		M_0=s, & \quad N_0=t, \\
		M_{n+1}= M(M_n, N_n),  &\quad   N_{n+1}= N(M_n, N_n),\quad n\in\bN_0.
	\end{aligned}
\ees
If both sequences converge to the same limit, this common value is called the Gaussian compound and is denoted by ${M\og N} (s,t)$. The best-known example is the arithmetic-geometric mean $A\og G$, obtained by iterating the arithmetic and geometric means.

Similarly, we can define the Archimedean compound $M\oa N \,(s,t)$ obtained by the iterative process 
\bes
	\begin{aligned}
	M_0=s, & \quad N_0=t, \\
	M_{n+1}= M(M_n, N_n), &  \quad   N_{n+1}= N(M_{n+1}, N_n),\quad n\in\bN_0,
	\end{aligned}
\ees
if both sequences converge to a common limit. A well-known example of such a mean is the Schwab-Borchardt mean. Related iterations are connected to the theory of elliptic functions; for details, see \cite{JJ}.

Asymptotic expansions have been proven useful in the study of means, especially when an explicit closed form is unknown or too complex to be practical. Continuing this line of work, we first extend the existing algorithms for the coefficients of the asymptotic expansions of the Gaussian and Archimedean compounds to the two-variable, nonsymmetric setting, allowing a joint analysis of both procedures.

		Recall that the (formal) series $\sum_{n=0}^{\infty} a_n \varphi_n(x)$ is said to be an \emph{asymptotic expansion} of a function $f(x)$ as $x\to x_0$,
			with respect to an asymptotic sequence $(\varphi_n(x))_{n\in \bN_0}$, 
			if, for each $N\in\bN_0$, the following holds (\cite{Erd}):
			$$
				f(x) = \sum_{n=0}^{N}a_n \varphi_n(x) +o(\varphi_N(x)).
			$$
In this paper, the asymptotic sequence is fixed from the beginning as
\[
\varphi_n(x)=x^{1-n},\qquad n\in\bN_0,\qquad x\to\infty.
\]
Thus, we consider means $M(s,t)$ that admit an asymptotic power series
expansion of the form
 	\bes
		M(x+s,x+t) \sim x\sum_{n=0}^\infty c_n(s,t)x^{-n},\quad x\to\infty.
	\ees

The algorithms obtained and analysed in \cite{BurElMih-2024} and
\cite{BurMih-Gauss} were derived for means whose asymptotic expansions in the variables $(x-t,x+t)$
contain only odd powers of $x$ (and even powers of $t$). They belong to the so-called symmetric case.
In order to analyse and compare the compounds starting from the same pair of means, we need to
obtain the more general, i.e. nonsymmetric case of the above mentioned algorithms. This case is
more complex than the already analysed one and requires more attention, especially regarding
the conditions on coefficients under which the functional equation (that defines the compound
mean) has the unique solution in the context of formal power series.

Therefore, the main goal of this paper is to extend the results for symmetric
means from \cite{BurElMih-2024,BurMih-Gauss} to the more general case of two
nonsymmetric means. We obtain recursive algorithms for the formal
coefficients of both Gaussian and Archimedean compounds. We identify the
singular cases and then discuss separately what the same coefficients say
about local convergence and how they are related to known global existence
conditions. The new algorithms are applied to classical nonsymmetric means,
particularly weighted power means. Finally, in the last section we use two
formal branches for sign-dependent expansions and study compounds of the
neo-Pythagorean (Greek) means and another recently introduced family.

\section{Preliminaries}

	A notion of asymptotic inequality can be introduced using asymptotic expansions, more precisely, the positivity of the first non-zero coefficient. We recall its definition.
	\begin{defn}[\cite{Vu}]\label{defn-AsymIneq}
	Let $F(s,t)$ be any homogeneous bivariate function such that
	\bes
		F(x+s,x+t)
		=c_{k}(s,t)x^{-k+1}+\O(x^{-k}).
	\ees
	If $c_{k}(s,t)>0$ for all $s$ and $t$, then
	we say $F$ is \textit{asymptotically} greater than zero, and write
	\bes
		F\succ 0.
	\ees
\end{defn}
The same paper explains that an asymptotic inequality is a necessary relation between comparable means. More precisely, if $F\ge0$, then $F\succ 0$. Furthermore, it is sufficient to consider the case $s=-t$ when studying asymptotic inequalities (\cite{El-2015-classical}).

Throughout this paper, we assume that the (homogeneous) means under consideration have, for each fixed $t$, asymptotic expansions as $x\to\infty$ of the following form:
 \begin{align}
 	M(x-t,x+t) &\sim \sum_{n=0}^\infty a_n t^n x^{-n+1}, \label{asymexp-M}\\
 	N(x-t,x+t) &\sim \sum_{n=0}^\infty b_n t^n x^{-n+1},\label{asymexp-N}\\
 	M\og N \,(x-t,x+t) &\sim \sum_{n=0}^\infty \c_n t^n x^{-n+1}, \label{asymexp-MogN} \\
 	M\oa N \,(x-t,x+t) &\sim \sum_{n=0}^\infty \a_n t^n x^{-n+1}. \label{asymexp-MoaN}
 \end{align}

After the usual substitution $u=t/x$, the expansion in the variables
$(x-t,x+t)$ may be treated as a formal power series in $u$. All calculations
are made in $\bR[[u]]$. In particular, equality of two formal expansions means
equality of their coefficient sequences; it does not assert equality of the
corresponding functions. 
 
\begin{rem}\label{rem-initialcoeffs}
The asymptotic inequalities that follow from the defining property of a mean give some useful consequences for the coefficients in the corresponding asymptotic expansion.
\begin{enumerate}
	\item \label{rem-coeff0} Because of the defining property of the mean \eqref{def-mean-minmax}, 
		observed in variables $(x-t,x+t)$ and consequently 
			because of the corresponding inequalities
			\bes
				x-t \le M(x-t,x+t) \le x+t
			\ees
		for $t>0$, and reversed for $t<0$,
		the initial coefficient in the asymptotic expansion of a mean is 
		$a_0=1$, and similarly $b_0=\c_0=1$.
	\item \label{rem-coeff1int} After setting $a_0=1$, the same sequence of asymptotic inequalities gives $a_1\in [-1,1]$.	
	\item \label{rem-coeff1pm1} 
	\label{rem-coeff1pm1} If $a_1=\pm1$, then the same reasoning yields $a_2=0$. 
	An endpoint value $a_1=\pm1$ does not, in the nonsymmetric setting,
force all higher coefficients ($n\ge3$) to vanish. 
\end{enumerate}

\end{rem}

\begin{rem}
In \cite{ElVu-2014-04,ElVu-2014-09}, the authors showed that the coefficients in the general form $M(x+s,x+t)$ can be easily derived from the expansion \eqref{asymexp-M} by using the variables
\bes
	\a=\frac{s+t}{2},\ \b=\frac{t-s}{2}.
\ees
Then we have
\bes
	M(x+s,x+t) =M(x+\a-\b,x+\a+\b) =
	\sum_{n=0}^{\infty}c_n(\a,\b)x^{-n+1},
\ees
where the polynomials $c_n(\a,\b)$ are calculated by
\bes
	c_n (\a,\b)= \sum_{k=0}^{n} a_{n-k} {k-n+1\choose k} \b^{n-k}\a^k.
\ees

\end{rem}

The following fundamental lemma on powers of asymptotic series is essential for our results.

\begin{lemma}[\cite{ChenElVu-2013,Gould-1974}]
	\label{lemma-power}
	Let 
	$$
		g(x)\sim \sum_{n=0}^\infty a_n x^{-n}
	$$
	be an asymptotic expansion of $g(x)$ as $x\to\infty$, with $a_0\neq0$. Then, for every real $r$, we have
	$$
		[g(x)]^{r}\sim \sum_{n=0}^\infty P[n,r,\ba]x^{-n},
	$$
	where $P[0,r,\ba]=a_0^r$ and
	\bes
	        P[n,r,\ba]=\frac1{na_0}\sum_{k=1}^n[k(1+r)-n]a_kP[n-k,r,\ba],\quad n\in\bN.
	\ees
\end{lemma}

The quantity $P[n,r,\ba]$ can be viewed as the coefficient of $x^{-n}$ in the $r$-th power of the series associated with the sequence $\ba$.
 
We set
\[
d_k^+(\ba,\bb)=b_k+a_k,\qquad k\in\bN_0.
\]
If the two coefficient sequences are distinct, let $z=z(\ba,\bb)$ be the
smallest integer $n\in\bN_0$ such that $a_n\ne b_n$, and define
\[
d_k^-(\ba,\bb)=b_{z+k}-a_{z+k},\qquad k\in\bN_0.
\]
Remark \ref{rem-initialcoeffs} (\ref{rem-coeff0}) gives $z\ge1$. If the two
coefficient sequences coincide, we set $z(\ba,\bb)=\infty$; in that case the
sequence $\bd^-$ is not needed.

We next present a key theorem on the asymptotic expansion of a composition of two
means. It will play a crucial role in our study and was proved in \cite{BurEl-2017};
the statement below is adapted to the notation introduced above, and the case
$z(\ba,\bb)=\infty$ is stated explicitly.

\begin{theoremA*} 
	\label{thm-composition}
	Let the means $M$ and $N$ have asymptotic expansions \eqref{asymexp-M} and \eqref{asymexp-N},
	and let the mean $K$ have an asymptotic expansion of the same type with the sequence of coefficients $\bc=(c_n)_{n\in\bN_0}$.
		If $z(\ba,\bb)<\infty$, then the composition $K(M,N)$ has an asymptotic expansion of the same type with the coefficients
	\bes
		C_m[\bc,\ba,\bb] = \frac12 \sum_{n=0}^{\lfloor \frac{m}{z(\ba,\bb)} \rfloor} 
			c_n \sum_{k=0}^{m-n\mkern1mu z(\ba,\bb)} P[k,n,\bd^-(\ba,\bb)] P[m-k-n\mkern1mu z(\ba,\bb),-n+1,\bd^+(\ba,\bb)], 
			\  m\in\bN_0.
		\ees
		If $z(\ba,\bb)=\infty$, then
		\[
			C_m[\bc,\ba,\bb]=a_m=b_m,\qquad m\in\bN_0.
		\]
\end{theoremA*}


This theorem was already used successfully to obtain the Gaussian and Archimedean compounds of symmetric means in \cite{BurElMih-2024,BurMih-Gauss}, and to define and analyse a new class of means in \cite{Mih-2023-coinciding}.

\section{Nonsymmetric Gaussian compound} \label{section-gauss}

Given the known asymptotic expansions of the means $M$ and $N$, our goal is to find a formula for the coefficients $\c_n$ in the asymptotic expansion \eqref{asymexp-MogN} of their Gaussian compound. The starting point is the functional equation
 	\be\label{rel-MogN}
		M\og N  = M\og N \, (M,N).
	\ee

\begin{thm}\label{thm-main-Gauss}
Assume that the means $M$, $N$ and $M\og N$ have asymptotic expansions
\eqref{asymexp-M}, \eqref{asymexp-N} and \eqref{asymexp-MogN}, respectively.
Then the coefficients $\c_n$ can be calculated as follows. 

If $z(\ba,\bb)=1$ and $d_0^-(\ba,\bb) \neq \pm2$, then $\gamma_0=1$ and
	\be\label{thm-gausscoeff-z=1}
		\begin{aligned}
		\c_m &= \frac1{ 2- (d_0^-(\ba,\bb))^m2^{1-m}} 
		\sum_{n=0}^{m-1}  \c_n \times \\
		&\qquad \qquad \times
			\sum_{k=0}^{m-n} P[k,n,\bd^-(\ba,\bb)]P[m-k-n,-n+1,\bd^+(\ba,\bb)],\quad m\in \bN.
		\end{aligned}
	\ee

If $1<z(\ba,\bb)<\infty$, then $\gamma_0=1$ and
	\be\label{thm-gausscoeff-z>1}
		\c_m = \frac12 \sum_{n=0}^{\lfloor \frac{m}{z(\ba,\bb)} \rfloor}  \c_n
			\sum_{k=0}^{m-n\mkern1mu z(\ba,\bb)} P[k,n,\bd^-(\ba,\bb)]P[m-k-n\mkern1mu z(\ba,\bb),-n+1,\bd^+(\ba,\bb)],
			\quad m\in \bN.
	\ee
	
If
$z(\ba,\bb)=\infty$, then $\gamma_m=a_m=b_m$ for every $m\in\bN_0$.

\end{thm}
	
\begin{proof}	
Assume that $z(\ba,\bb)<\infty$.
	According to Theorem A, the coefficients of $t^m x^{-m+1}$ in the asymptotic expansion of the right-hand side of \eqref{rel-MogN} are equal to $C_m [\pmb{\gamma}, \ba,\bb]$.
Comparison with the asymptotic expansion \eqref{asymexp-MogN} of the left-hand side of \eqref{rel-MogN} gives
	\be\label{coeff-gamma-gauss-general}
		\c_m  = \frac12 \sum_{n=0}^{\lfloor \frac{m}{z(\ba,\bb)} \rfloor}  \c_n
			\sum_{k=0}^{m-n\mkern1mu z(\ba,\bb)} P[k,n,\bd^-(\ba,\bb)]P[m-k-n\mkern1mu z(\ba,\bb),-n+1,\bd^+(\ba,\bb)]. 
	\ee

In case $z>1$, formula \eqref{coeff-gamma-gauss-general} already gives an explicit expression for the coefficient $\c_m$, namely \eqref{thm-gausscoeff-z>1}.

In case $z=1$, the coefficient $\c_m$ also appears on the right-hand side of
\eqref{coeff-gamma-gauss-general}:
	\begin{align*}
		\c_m &= \frac12 \sum_{n=0}^{m}  \c_n
			\sum_{k=0}^{m-n} P[k,n,\bd^-(\ba,\bb)]P[m-k-n,-n+1,\bd^+(\ba,\bb)] \\
			&= \frac12 \sum_{n=0}^{m-1}  \c_n
			\sum_{k=0}^{m-n} P[k,n,\bd^-(\ba,\bb)]P[m-k-n,-n+1,\bd^+(\ba,\bb)] \\
			& \qquad\qquad+ \frac12 (d_0^-(\ba,\bb))^m  (d_0^+(\ba,\bb))^{1-m} \c_m.
	\end{align*}
Since $d_0^+(\ba,\bb) =2$ and $d_0^-(\ba,\bb) \neq \pm 2$,
we obtain the expression \eqref{thm-gausscoeff-z=1}.
%
	If $z(\ba,\bb)=\infty$, Theorem A gives
	$C_m[\pmb{\gamma},\ba,\bb]=a_m=b_m$. Comparing both sides of
	\eqref{rel-MogN} proves the last assertion.
\end{proof}

\begin{rem}\label{rem-a1b1pm1}
	Suppose that $z(a,b)=1$ and
\[
d^-_0(a,b)=b_1-a_1=\pm2.
\]
Since $a_1,b_1\in[-1,1]$, the only possible pairs of first coefficients
are
\[
(a_1,b_1)\in\{(-1,1),(1,-1)\}.
\]
In this case the coefficient multiplying $\gamma_m$ in the formal
invariance equation vanishes, and formula \eqref{thm-gausscoeff-z>1} cannot be
used to determine $\gamma_m$ recursively.
The existence and
convergence of the corresponding functional iteration must be studied
separately. 

For example, the projection means belong to this case:
	their expansions are projection series $\Pi_1(x-t,x+t)=x-t$ and  $\Pi_2(x-t,x+t)=x+t$, and their Gaussian compound does not exist. Here the
	singularity of the formal recursion reflects a genuine non-existence of the
	Gaussian compound. 
	
The same pair of first coefficients, however, does not by
	itself preclude existence.
	Consider, for instance, means
	\[
		P(s,t)=t-\frac{(t-s)^3}{2(s+t)^2},\qquad
		Q(s,t)=s+\frac{(t-s)^3}{2(s+t)^2},
	\]
	whose two-sided expansions of type \eqref{asymexp-M} are
	\[
		P(x-t,x+t)=x+t-t^3x^{-2},\qquad
		Q(x-t,x+t)=x-t+t^3x^{-2},
	\]
	so that $(a_1,b_1)=(1,-1)$ and $d_0^-(\ba,\bb)=-2$, exactly the exceptional pair
	above, while neither mean is a projection. Since $P(s,t)+Q(s,t)=s+t$, the
	Gaussian iteration preserves the sum, and hence the compound exists and
	$ 		P\og Q=A.$
	
\end{rem}

In the case $z(\ba,\bb)=1$ and $d_0^-(\ba,\bb) \neq \pm2$, the first few coefficients are:
\begin{align*}
	\gamma_0 &=1,\\
	\gamma_1 &= \frac{d_1^+(\ba,\bb)}{2-d_0^+(\ba,\bb)} =\frac{a_1+b_1}{2+a_1-b_1},\\
	\gamma_2 &= \frac{2\left( d_1^-(\ba,\bb) \, d_1^+(\ba,\bb) 
		+d_2^+(\ba,\bb)(2-d_0^-(\ba,\bb))\right)}{(4-d_0^-(\ba,\bb)^2)(2-d_0^-(\ba,\bb))}
		= \frac{4\, (a_2-a_2 b_1+b_2+a_1 b_2)}{(2+b_1-a_1)(2+a_1-b_1)^2},
\end{align*}
and in the case $z(\ba,\bb)=2$ we have:
\begin{align*}
	\gamma_0 &=1,\\
	\gamma_1 &= \tfrac12 d_1^+(\ba,\bb) = \tfrac12(a_1+b_1) ,\\
	\gamma_2 &= \tfrac14 \left( d_0^-(\ba,\bb) \, d_1^+(\ba,\bb)+2\, d_2^+(\ba,\bb) \right)
		=\tfrac14 \left( -a_2 (-2 + a_1 + b_1) + (2 + a_1 + b_1) b_2 \right),\\
	\gamma_3 &= \tfrac14 \left( d_1^-(\ba,\bb) \, d_1^+(\ba,\bb)+2\, d_3^+(\ba,\bb) \right)
		= \tfrac14\left( -a_3 (-2 + a_1 + b_1) + (2 + a_1 + b_1) b_3 \right),\\
	\gamma_4 &= \tfrac{1}{16} ((d_0^-(\ba,\bb)^3+4 d_2^-(\ba,\bb)) d_1^+(\ba,\bb)+2 d_0^-(\ba,\bb)^2 d_2^+(\ba,\bb)+8 d_4^+(\ba,\bb)),
\end{align*}

\begin{rem} Since we use the same algorithm \eqref{thm-gausscoeff-z>1} for every finite $z>1$, the coefficients for $z>2$ need not be computed from separate formulas. Instead, one first evaluates them symbolically from the case $z=2$ and then rearranges the resulting expressions by means of a purely formal substitution of the symbols $d_k^-(\ba,\bb)$: for $k<z-2$ the symbol $d_k^-(\ba,\bb)$ is replaced by $0$, while for $k\ge z-2$ the symbol $d_k^-(\ba,\bb)$ is replaced by $d_{k-(z-2)}^-(\ba,\bb)$, i.e.\ by the coefficient of the given sequence indexed accordingly. This is only a relabelling of the terms in the $z=2$ expressions and does not equate $d_k^-(\ba,\bb)$ with zero, which would contradict its definition. For example, in the case $z=3$ the substitution replaces $d_0^-(\ba,\bb)$ by $0$ and shifts the remaining symbols, and the coefficients above give
\begin{align*}
	\gamma_0 &=1,\\
	\gamma_1 &= \tfrac12 d_1^+(\ba,\bb) = \tfrac12(a_1+b_1) ,\\
	\gamma_2 &=\tfrac12 d_2^+(\ba,\bb) = \tfrac12(a_2+b_2), \\
	\gamma_3 &= \tfrac14 \left( d_0^-(\ba,\bb) \, d_1^+(\ba,\bb)+2\, d_3^+(\ba,\bb) \right),\\
	\gamma_4 &= \tfrac{1}{4}\left(d_1^-(\ba,\bb)d_1^+(\ba,\bb)+2 d_4^+(\ba,\bb)\right).
\end{align*}
\end{rem}


\section{Nonsymmetric Archimedean compound}\label{section-arch}

	In this section, we obtain the general asymptotic expansion of the Archimedean compound of two means. Let $M$ and $N$ be means with asymptotic expansions \eqref{asymexp-M} and \eqref{asymexp-N}, and let their Archimedean compound have the asymptotic expansion \eqref{asymexp-MoaN}.

We find the coefficients $\a_n$ by solving the functional equation
\be\label{rel-MoaN} 
	M \oa N \,(s,t)=M\oa N\, (M(s,t),N^*(s,t))
\ee
in terms of asymptotic expansions, where $N^*(s,t)=N(M(s,t),t)$.
 Thus, the Archimedean iteration of $(M,N)$ is the Gaussian iteration
of the pair $(M,N^*)$. Depending on the first index at which the
formal series of $M$ and $N^*$ differ, all three cases of
Theorem~\ref{thm-main-Gauss} may occur.

 \begin{thm}\label{thm-main-Arch}
 Let $M$, $N$ and $M\oa N$ be means with asymptotic expansions \eqref{asymexp-M}, \eqref{asymexp-N} and \eqref{asymexp-MoaN}, respectively.
Define ${\bf e}=(1,1,0,0,\ldots)$ and $c_m=C_m[\bb,\ba,{\bf e}]$, $m\in\bN_0$,
so that $\bc=(c_m)_{m\in\bN_0}$ is the coefficient sequence of
$
N^*(s,t)=N(M(s,t),t).
$

If $z(\ba,\bc)=1$ and $(a_1,b_1)\neq(-1,1)$, then
$\alpha_0=1$ and 
	\be
	\begin{aligned}
		\a_m &= \frac{1}{2} \left[ 1-\frac{(1-a_1)^m (1+b_1)^m}{2^{2m}}\right]^{-1}
		 \sum_{n=0}^{m-1} \a_n \times \\
		 	&\qquad \qquad \times \sum_{k=0}^{m-n} 
		 	 P[k,n,\bd^-(\ba,\bc)] \, P[m-n-k,-n+1,\bd^+(\ba,\bc)],
		\quad m\in\bN.
	\end{aligned}
	\ee

If $1<z(\ba,\bc)<\infty$, then $\alpha_0=1$ and 
\[
\alpha_m=
\frac12
\sum_{n=0}^{\lfloor \frac{m}{z(\ba,\bc)} \rfloor}
\alpha_n
\sum_{k=0}^{m-n\mkern1mu z(\ba,\bc)}
P[k,n,\bd^-(\ba,\bc)]
P[m-k-n\mkern1mu z(\ba,\bc),-n+1,\bd^+(\ba,\bc)],
\quad m\in\bN.
\]

If $z(\ba,\bc)=\infty$, then
\[
\alpha_m=a_m=c_m,\quad m\in\bN_0.
\]

For $(a_1,b_1)=(-1,1)$, the branch with
$d^-_0(\ba,\bc)=2$ is singular and the formal invariance equation does
not determine the coefficients recursively.
  \end{thm}

 \begin{proof}
	According to Theorem A, the composition $N^*$ has the asymptotic expansion
 	\begin{align*}
		N^*  (x-t,x+t) &= N(M(x-t,x+t),x+t)  
		   \sim \sum_{m=0}^\infty  C_m [\bb,\ba,\mathbf{e}] \, t^m x^{-m+1}
		   =  \sum_{m=0}^\infty  c_m \, t^m x^{-m+1}.
	\end{align*} 

The Archimedean iteration of $(M,N)$ is precisely the Gaussian
iteration of the pair $(M,N^*)$, since the invariance equation
for $M\otimes_a N$ is
\[
M\otimes_a N
=
M\otimes_a N \,(M,N^*).
\]
The coefficient sequence $(\alpha_n)_{n\in\bN_0}$ is therefore obtained by applying
Theorem~\ref{thm-main-Gauss} to the sequences $\ba$ and $\bc$.

It remains only to identify the first coefficient of their difference.
A direct computation gives
\[
c_1=
\frac{1+a_1+(1-a_1)b_1}{2},
\]
and hence
\[
c_1-a_1=
\frac{(1-a_1)(1+b_1)}{2}.
\]

If $(1-a_1)(1+b_1)>0$, then $z(\ba,\bc)=1$.
The value $c_1-a_1=2$ occurs only when
$ (a_1,b_1)=(-1,1)$,
which is the singular case. Otherwise, formula \eqref{thm-gausscoeff-z=1}
applied to $(\ba,\bc)$ gives the stated recursion.

If
 $ (1-a_1)(1+b_1)=0$,
that is, if $a_1=1$ or $b_1=-1$, then $c_1=a_1$.
In this situation $z(\ba,\bc)$ may be any integer greater than one or may
be infinite. The corresponding case of Theorem~\ref{thm-main-Gauss} must therefore be
applied.

This proves all the assertions.
	 \end{proof}

If $a_1\neq1$, $b_1\neq-1$ and $(a_1,b_1)\neq(-1,1)$, the first few coefficients are:
\begin{align*}
	\alpha_0 &= 1,\\
	\alpha_1 &= \frac{2 d_1^+(\ba,\bc)}{3+a_1+(a_1-1)b_1} 
		=\frac{1 + 3 a_1 + (1-a_1) b_1}{3+a_1+(a_1-1)b_1},\\
	\alpha_2 &= \frac{2 d_1^-(\ba,\bc) d_1^+(\ba,\bc) + 
 d_2^+(\ba,\bc)(3 + a_1 + (-1 + a_1) b_1)}%
	{ (2 - \frac{1}{8} (-1 + a_1)^2 (1 + b_1)^2)(3 + a_1 + (-1 + a_1) b_1)}\\
	&= \frac{32 a_2 (-1 + b_1) - 8 (-1 + a_1)^2 (1 + a_1) b_2}%
		{(-5 + a_1 + (-1 + a_1) b_1) (3 + a_1 + (-1 + a_1) b_1)^2}.
\end{align*}

\section{Formal coefficients, metric contraction, and existence}\label{section-formal}

The main results in Sections \ref{section-gauss} and \ref{section-arch} solve invariance equations in the formal
algebra fixed in the Introduction. There are, however, three different
questions: whether the iterative procedure converges, whether its limit has
an expansion in the scale $x^{1-n}$, and whether the formal invariance
equation determines the coefficients of that expansion. Our recursions answer
the third question under the assumptions made in the first two. In this
section we give some additional conditions which guarantee convergence and
show how they are reflected in the formal coefficients.

Classical existence results for compound means go back to Foster and Phillips
\cite{Foster} and were later refined by Costin and Toader
\cite{CostinToader}; see also the overview \cite{ToaderCostin}. 

We use the metric introduced by Farhi \cite{Farhi} for symmetric strict means and later studied, in a related homogeneous setting, by Witkowski \cite{Witk-2015}. The same formula defines a useful distance for the arbitrary bivariate means considered here.

 For two bivariate means $M$ and $N$, let
\begin{equation}\label{metric-d}
	d(M,N)=\sup_{s\neq w}\frac{|M(s,w)-N(s,w)|}{|w-s|}.
\end{equation}
Since both values lie between $s$ and $w$, one always has $0\leqslant d(M,N)\leqslant1$.

For homogeneous means this supremum can be written in the variables used throughout
the paper. Every pair $s\neq w$ of positive numbers can be written as
$s=x(1-u)$, $w=x(1+u)$ with
\[
	x=\frac{s+w}{2},\qquad u=\frac{w-s}{w+s}=\frac{t}{x},\qquad 0<|u|<1,
\]
where $t=\frac{w-s}{2}$ is the half-gap appearing in \eqref{asymexp-M}. Since
$|w-s|=2x|u|$ and, by homogeneity, $M(s,w)=x\,M(1-u,1+u)$, the quotient in
\eqref{metric-d} depends on $u$ alone, and
\begin{equation}\label{metric-normalized}
	d(M,N)=\sup_{0<|u|<1}\frac{|M(1-u,1+u)-N(1-u,1+u)|}{2|u|}.
\end{equation}
This normalization also makes the relation with our formal series direct. Indeed,
dividing \eqref{asymexp-M} by $x$ turns it into
\[
	M(1-u,1+u)\sim\sum_{n=0}^{\infty}a_nu^n=1+a_1u+\mathcal{O}(u^2),
	\qquad u\to0,
\]
and likewise $N(1-u,1+u)=1+b_1u+\mathcal{O}(u^2)$, so that the coefficients
appearing here are precisely those of \eqref{asymexp-M} and \eqref{asymexp-N}.
Letting $u\to0$ in \eqref{metric-normalized} gives
\begin{equation}\label{metric-lower-gauss}
d(M,N)\ge q_g:=\frac{|a_1-b_1|}{2}.
\end{equation}
No better bound can be stated in terms of $a_1$ and $b_1$ alone.
In fact, for the
weighted arithmetic means
$A_\lambda(s,u)=\lambda s+(1-\lambda)u$, one has
\[
d(A_\lambda,A_\mu)=|\lambda-\mu|
=\frac{|a_1-b_1|}{2}.
\]

Farhi proved that $d(M,N)<1$ guarantees the existence and uniqueness of the
invariant mean and the convergence of the Gaussian iteration
\cite[Theorem 4.1]{Farhi}, with the error estimate
\be \label{metric-error-gauss}
	|M_n-L|\leqslant k^n|s-u|,\qquad |N_n-L|\leqslant k^n|s-u|,\qquad k=d(M,N).
\ee
Farhi worked with symmetric strict means, but under our wider convention, in which a mean is
only required to satisfy \eqref{def-mean-minmax}, the same nested interval argument gives the same conclusion. We
use this known contraction criterion only as background for the coefficient estimates below.

The first coefficients give a local version of this result even when the
global distance is not known. For the Gaussian iteration, $q_g$ in
\eqref{metric-lower-gauss} is the limiting ratio of the output gap to the
input gap near the diagonal. Hence $q_g<1$ makes the diagonal locally
attracting. More precisely, if $z=z(\ba,\bb)\ge2$ is finite and an iteration
converges to $L>0$, then, with $\Delta_n=|M_n-N_n|$,
\begin{equation}\label{local-order-gauss}
\Delta_{n+1}\sim
\frac{|a_z-b_z|}{2^zL^{z-1}}\Delta_n^z.
\end{equation}
Thus $z$ is also the local order of convergence. If $z=1$, the asymptotic
linear factor is $q_g$. If $z=\infty$, the two formal series agree to every
algebraic order, so the actual rate depends on terms which are invisible in
the fixed formal scale.

The Archimedean procedure can be treated by the same metric after one simple
observation. Put
\bes
N^*(s,u)=N(M(s,u),u).
\ees
Then the Archimedean iteration of $(M,N)$ is exactly the Gaussian iteration
of $(M,N^*)$. Consequently,
\bes
d(M,N^*)<1
\ees
is a sufficient condition for the existence and uniqueness of $M\oa N$, with
the error estimate \eqref{metric-error-gauss} in which $k=d(M,N^*)$.
The first coefficient $c_1$ of $N^*$ is
\[
c_1=\frac{1+a_1+(1-a_1)b_1}{2},
\qquad
c_1-a_1=\frac{(1-a_1)(1+b_1)}{2}.
\]
It follows that
\bes
d(M,N^*)\ge q_a:=\frac{(1-a_1)(1+b_1)}{4}.
\ees
Since $a_1,b_1\in[-1,1]$, we have $0\le q_a\le1$, and $q_a=1$ only for
$(a_1,b_1)=(-1,1)$, exactly the singular case of the formal Archimedean
recursion. The bound is again sharp: if $M=A_\lambda$ and $N=A_\mu$, then
$N^*=A_{\lambda\mu}$ and
\[
d(M,N^*)=\lambda(1-\mu)=q_a.
\]
In particular, $q_a<1$ gives local convergence near the diagonal. Higher
local orders are obtained from the first coefficient at which the formal
series of $M$ and $N^*$ differ, just as in
\eqref{local-order-gauss}.


We can now state precisely what is detected by the exceptional formal
coefficients. In the Gaussian case, $|a_1-b_1|=2$ forces $d(M,N)=1$ and
removes the linear contraction. In the Archimedean case,
$(a_1,b_1)=(-1,1)$ similarly gives $q_a=1$. At these boundary points,
higher-order or formally flat terms decide the functional convergence. The
formal recursion records the loss of recursive coefficient determination,
while the metric and the direct iteration answer the separate existence
question.


\section{Compounds of the weighted power means}\label{section-weightedPower}

In this section, we give some examples of compounds of nonsymmetric means. In particular, we apply our theorems to weighted power means. Recall that the weighted bivariate power mean with weights $(\lambda,1-\lambda)$, where $\lambda\in(0,1)$, is defined by
\[
M_{p,\lambda}(s,t)=[\lambda s^p + (1-\lambda)t^p]^{\frac1p},\quad p\neq0,
\]
where the limit case $p\to 0$ corresponds to the geometric mean
\[
G_{\lambda}(s,t)= s^\lambda\,t^{1-\lambda}.
\]
In \cite{ElVu-2016-01}, the authors derived an asymptotic expansion of the general $n$-variate power mean. Applying their Theorem 2.1 to our case gives the following result.

\begin{cor}
The weighted bivariate power mean has the following asymptotic expansion as $x\to\infty$:
\[M_{p,\lambda}(x-t,x+t)\sim\sum_{k=0}^{\infty} c_k(p,\lambda) t^k x^{-k+1},\]
where $c_0=1$ and
\[
c_k(p,\lambda)=\frac1k  \sum_{j=1}^k\left(j(1+\tfrac1p)-k\right)\binom{p}{j}\left(1+\lambda\left((-1)^j-1\right)\right)c_{k-j}(p,\lambda).
\]

\end{cor}

The first few coefficients in the expansion of the power mean are:
\bes
\begin{split}
	c_0 &= 1,\\
	c_1 &= 1-2\lambda,\\
	c_2 &= -2\lambda(\lambda-1)(p-1),\\
	c_3&=-\tfrac23 \lambda(\lambda-1)(2\lambda-1)(p-1)(2p-1).
\end{split}
\ees
Note that the coefficients for the weighted geometric mean follow directly from $c_k(p,\lambda)$ by letting $p\to 0$; see Lemma 2.2 in \cite{ElVu-2014-04}.

Our algorithms give the first few coefficients for the Gaussian and Archimedean compounds of $M_{p,\lambda}$ and $M_{r,\mu}$. Their expansions are as follows:
\bes
\begin{split}
M_{p,\lambda}\og& M_{r,\mu} \,(x-t,x+t)=x+\frac{\lambda + \mu-1}{\lambda - \mu-1}t+\\
&+\frac{2\mu (1 - \lambda) ( \lambda(p-1)+\mu(1-r) + r -1)}{(1 + \lambda - \mu) (1 - \lambda +
     \mu)^2}t^2x^{-1}+\dots
\end{split}
\ees

\bes
\begin{split}
M_{p,\lambda}\oa& M_{r,\mu} \,(x-t,x+t)=x+\frac{1-\lambda(1+ \mu)}{1+\lambda(-1 +\mu)}t+\\
&+\frac{2\lambda\mu (1 - \lambda) ( 1-p+\lambda(\mu-1)(r-1))}{(-1 + \lambda(-1+\mu)) (1 +\lambda(-1+
     \mu))^2}t^2x^{-1}+\dots
\end{split}
\ees

For example, if $p=1$ and $r=0$, we obtain the compounds of the weighted arithmetic and weighted geometric means:

\[\begin{split}
A_{\lambda}\og& G_{\mu} \,(x-t,x+t)=x+\frac{\lambda + \mu-1}{\lambda - \mu-1}t+\\
&+\frac{2(1 - \lambda)\mu ( \mu-1)}{(1 + \lambda - \mu) (1 - \lambda +
     \mu)^2}t^2x^{-1}+\dots
\end{split}
\]

\[\begin{split}
A_{\lambda}\oa& G_{\mu} \,(x-t,x+t)=x+\frac{1-\lambda(1+ \mu)}{1+\lambda(-1 +\mu)}t+\\
&+\frac{2\lambda^2 ( \lambda-1) \mu(\mu-1)}{(-1 + \lambda(-1+\mu)) (1 +\lambda(-1+
     \mu))^2}t^2x^{-1}+\dots
\end{split}
\]

Also note that for $\lambda=\mu=\frac12$, we obtain the known asymptotic expansion of the arithmetic-geometric mean, derived in \cite{BurEl-2015,BurMih-Gauss} for the Gaussian compound and in \cite{BurElMih-2024} for the Archimedean compound.
\[ 
A \og G \, (x-t,x+t) = x-\frac14 t^2x^{-1} - \frac{5}{64}t^4x^{-3} - \frac{11}{256}t^6x^{-5} - \dots
\]
\[ 
A \oa G \,(x-t,x+t) = x+\frac13t-\frac4{45}t^2x^{-1} + \frac{44}{945}t^3x^{-2}- \frac{428}{14175}t^4x^{-3} - \dots
\]

Using asymptotic expansions of compound means, we can test whether a
Gaussian or Archi\-medean compound can belong to a proposed class of means. A
mismatching coefficient disproves such membership, whereas matching
coefficients may suggest a candidate that must then be verified by an
independent functional argument. The following known identities provide a
consistency check for our coefficient recursions; see \cite{ToaderCostin}.

\begin{cor}\label{power-cor} For every $p\in\mathbb{R}$ and $\lambda,\mu\in(0,1)$, the following identities hold:
\[M_{p,\lambda}\og M_{p,\mu} \,(s,t)= M_{p,\frac{\mu}{1-\lambda+\mu}}(s,t)\]
\[
M_{p,\lambda}\oa  M_{p,\mu} \,(s,t)=M_{p,\frac{\lambda\mu}{1-\lambda+\lambda\mu}}(s,t)\]
\[
M_{p,\lambda}\og M_{-p,1-\lambda} \,(s,t) = M_{0,\frac12}(s,t)=G(s,t)
\]

\end{cor}

Thus coefficient comparison in this example verifies consistency with exact
identities. More generally, it is a useful constructive tool for finding
candidates and for ruling them out, but it does not establish equality of
functions without an additional uniqueness or invariance argument.

\bigskip

\section{Branchwise expansions and neo-Pythagorean means}\label{section-moregeneral}

The neo-Pythagorean means were defined in \cite[p.\ 401]{Bull-2003}; some authors simply call them Greek means (see \cite{CostinToader}). The first four are the classical arithmetic ($A$), geometric ($G$), harmonic ($H$), and contraharmonic ($cH$) means. The others, for $0<a<b$, are defined by
\begin{align*}
	B_5(a,b)&=\frac12 \left[ b-a+ \sqrt{(b-a)^2+4a^2} \right], 
	&B_6(a,b)&=\frac12 \left[ a-b+ \sqrt{(b-a)^2+4b^2} \right],\\
	B_7(a,b)&= a+\frac{(b-a)^2}{b},
	&B_8(a,b)&=b-\frac{(b-a)^2}{b},\\
	B_9(a,b)&=\frac{b^2}{2b-a},
	&B_{10}(a,b)&=\frac12 \left[a+ \sqrt{a^2+4a(b-a)} \right].
\end{align*}
The definitions are completed by symmetry, $B_i(a,b)= B_i(b,a)$, together with $B_i(a,a)=a$.

Their asymptotic expansions can be easily obtained. For example, for $t>0$ and $x$ sufficiently large, we have:
\begin{align*}
	B_5(x-t,x+t)&=t+ \sqrt{t^2+(x-t)^2} 
		=t+ x \sqrt{1-2\tfrac{t}{x}(1-\tfrac{t}{x})}\\
		&=x+\sum_{n=2}^\infty (-1)^n \left[\sum_{k=0}^n\binom{\frac12}{k}\binom{k}{n-k} 2^k\right] t^n x^{-n+1},\\
	B_7(x-t,x+t)&= x-t+\frac{4t^2}{x+t} 
		 = x-t+4t^2x^{-2}\sum_{n=0}^\infty (-1)^n t^n x^{-n+1},\\
\end{align*}
The other expansions are omitted because they are not needed below and are
obtained by the same calculation.
For $t<0$, we have the same expansions as above with $-t$ in place of $t$. Hence, the general expansions, valid for all $t$, can be written as
	\bes
	 	B_5(x-t,x+t) = x+ \frac12t^2 x^{-1} +\frac12 \lvert t \rvert^3 x^{-2} 
	 		+  \frac38  t^4 x^{-3} +\frac18 \lvert t \rvert^5 x^{-4} 
	 		+\ldots
\ees
\bes	
	 	B_7(x-t,x+t) = x- \lvert t \rvert +4t^2 x^{-1} -4 \lvert t \rvert^3 x^{-2} 
	 		+  4  t^4 x^{-3} - 4 \lvert t \rvert^5 x^{-4} 
	 		+\ldots
	\ees

Note that the asymptotic expansions of these means cannot be written in the form \eqref{asymexp-M}. This motivates the slightly more general type of expansion introduced below.

Consider now the \textbf{branchwise expansion} of the following type:
\be \label{asymexp-M-abs-t}
	M(x-t,x+t) \sim \sum_{n=0}^{\infty} a_n(t) \, x^{-n+1}, 
\quad \text{ where } 
a_n(t)= 
	\begin{cases} \overline{a}_n t^n, & t>0,\\
					\underline{a}_n t^n, & t<0.
	\end{cases}
\ee 
When $\overline{a}_n=\underline{a}_n=a_n$ for each $n\in\bN_0$ in \eqref{asymexp-M-abs-t}, we obtain expansion \eqref{asymexp-M}. 

Thus \eqref{asymexp-M-abs-t} gives two formal coefficient series,
$\sum_{n\ge0}\overline a_nu^n$ for $t>0$ and
$\sum_{n\ge0}\underline a_nu^n$ for $t<0$. The branchwise calculations
below are again made in $\bR[[u]]$, separately for the two signs of $t$.

\begin{rem}	 
	 For each branch of \eqref{asymexp-M-abs-t}, the analogues of parts \eqref{rem-coeff0} and \eqref{rem-coeff1int} of Remark~\ref{rem-initialcoeffs} hold. In particular, the constant coefficient is $1$
and the first coefficient belongs to $[-1,1]$.  

An endpoint first coefficient $1$ or $-1$ does not force the higher coefficients of that branch to vanish. In fact, as explained in
Remark \ref{rem-initialcoeffs}, this conclusion is not valid even for a two-sided
nonsymmetric expansion. Therefore, whenever a branch has an endpoint
first coefficient together with nonzero higher coefficients, its full
formal series must be retained. The branch series of $B_7$ and $B_8$ provide natural examples of this
case.
	 
\end{rem}

We next explain how Theorems \ref{thm-main-Gauss} and \ref{thm-main-Arch} can be used for this type of expansion. Denote $\overline{\ba} = (\overline{a}_n)_{n\in\bN_0}$,
$\underline{\ba} = (\underline{a}_n)_{n\in\bN_0}$, 
and similarly for other sequences.

\begin{thm}\label{thm-g-pm}
Assume that the means $M$, $N$, and $M \og N$ have asymptotic expansions of type \eqref{asymexp-M-abs-t} with sequences of coefficients $(a_n(t))_{n\in\bN_0}$, $(b_n(t))_{n\in\bN_0}$, and $(\gamma_n(t))_{n\in\bN_0}$, respectively, and assume that $M<N$. Then, the coefficients $\overline{\gamma}_n$ are obtained by applying Theorem \ref{thm-main-Gauss} to the sequences $\ba=\overline{\ba}$ and $\bb=\overline{\bb}$, while the coefficients $\underline{\gamma}_n$ are obtained from the same theorem by taking $\gamma_n=\overline{\gamma}_n$, $\ba=\underline{\ba}$, and $\bb=\underline{\bb}$. If $N<M$, the coefficients $\gamma_n$ can be obtained analogously.
\end{thm}

\begin{proof}
Let $M<N$. For $t>0$, the left-hand side of \eqref{rel-MogN} has coefficients 
$(\overline{\gamma}_n)_{n\in\bN_0}$, while the corresponding expansions on the right-hand side have coefficients $(\overline{\gamma}_n)_{n\in\bN_0}$, $(\overline{a}_n)_{n\in\bN_0}$, and $(\overline{b}_n)_{n\in\bN_0}$. Hence, $\overline{\gamma}_n$ can be obtained from \eqref{thm-gausscoeff-z>1} or \eqref{thm-gausscoeff-z=1}, depending on $z=z(\overline{\ba},\overline{\bb})$, by taking $\ba=\overline{\ba}$, $\bb=\overline{\bb}$, and $\gamma_n=\overline{\gamma}_n$.
For $t<0$, the left-hand side of \eqref{rel-MogN} has coefficients 
$(\underline{\gamma}_n)_{n\in\bN_0}$. Because of the relation between $M$ and $N$, the corresponding expansions on the right-hand side have coefficients $(\overline{\gamma}_n)_{n\in\bN_0}$, $(\underline{a}_n)_{n\in\bN_0}$, and $(\underline{b}_n)_{n\in\bN_0}$. Hence, $\underline{\gamma}_n$ can be obtained from \eqref{coeff-gamma-gauss-general} by taking $\gamma_n=\underline{\gamma}_n$ on the left-hand side, and $\ba=\underline{\ba}$, $\bb=\underline{\bb}$, and $\gamma_n=\overline{\gamma}_n$ on the right-hand side.
Now let $M>N$. First, for $t<0$, we compute the coefficients $\underline{\gamma}_n$ from \eqref{thm-gausscoeff-z>1} or \eqref{thm-gausscoeff-z=1}, depending on $z=z(\underline{\ba},\underline{\bb})$, by taking $\ba=\underline{\ba}$, $\bb=\underline{\bb}$, and $\gamma_n=\underline{\gamma}_n$. Then, to compute $\overline{\gamma}_n$ for $t>0$, we use formula \eqref{coeff-gamma-gauss-general}, with $\gamma_n=\overline{\gamma}_n$ on the left-hand side and $\ba=\overline{\ba}$, $\bb=\overline{\bb}$, and $\gamma_n=\underline{\gamma}_n$ on the right-hand side.
\end{proof}

\begin{thm}\label{thm-a-pm}
Assume that the means $M$, $N$, and $M \otimes_a N$ have asymptotic expansions of type \eqref{asymexp-M-abs-t} with sequences of coefficients $(a_n(t))_{n\in\bN_0}$, $(b_n(t))_{n\in\bN_0}$, and $(\alpha_n(t))_{n\in\bN_0}$, respectively. Then, the coefficients $\overline{\alpha}_n$ are obtained by applying Theorem \ref{thm-main-Arch} to the sequences $\ba=\overline{\ba}$ and $\bb=\overline{\bb}$, while the coefficients $\underline{\alpha}_n$ are obtained from the same theorem by taking $\ba=\underline{\ba}$ and $\bb=\underline{\bb}$.
\end{thm}

\begin{proof}For $t>0$, the asymptotic expansion on the left-hand side of \eqref{rel-MoaN} has coefficients $\overline{\alpha}_n$. Furthermore, the mean $M$ in the composition $N^\ast(x-t,x+t) = N (M (x-t,x+t), x+t)$ has coefficients $\overline{a}_n$ in its asymptotic expansion. Since $M(x-t,x+t)\le x+t$, the mean $N$ in the same composition has coefficients $\overline{b}_n$. Also, since
$M (x-t,x+t)= \min(M (x-t,x+t),x+t) \le N (M (x-t,x+t), x+t)$, the mean $M\otimes_a N$ on the right-hand side of \eqref{rel-MoaN} has coefficients $\overline{\alpha}_n$. Hence, to obtain the coefficients in the asymptotic expansion of the Archimedean compound in this case, we may use Theorem \ref{thm-main-Arch} with $\alpha_n=\overline{\alpha}_n$, $\ba=\overline{\ba}$, and $\bb=\overline{\bb}$. If $t<0$, an analogous argument shows that we may use Theorem \ref{thm-main-Arch} with $\alpha_n=\underline{\alpha}_n$, $\ba=\underline{\ba}$, and $\bb=\underline{\bb}$.
\end{proof}

For demonstration, application of the previous theorems to the means $B_5$ and $B_6$ yields the following expansions:
\bes
B_5 \og B_6 \,(x-t,x+t)=x+\frac12 t^2x^{-1} + \frac38 t^4x^{-3} - \frac{1}{16} t^6x^{-5}+\cdots 
\ees
\bes 
B_5 \oa B_6 \,(x-t,x+t)=x+\frac13 |t| + \frac49 t^2x^{-1} + \frac{4}{27} |t|^3x^{-2}+\frac{116}{405}t^4x^{-3}+\cdots 
\ees
Since the definitions of $B_5$ and $B_6$ are completed by symmetry, their Gaussian compound $B_5\og B_6$ is also symmetric (it has only even powers of $t$) and therefore equals $B_6\og B_5$.

\begin{rem}\label{rem-branch} The exceptional cases of Theorems~\ref{thm-main-Gauss} and \ref{thm-main-Arch} are understood branchwise.

For Theorem~\ref{thm-g-pm}, if $M<N$ and the relevant first discrepancy satisfies
$z=1$, the positive branch requires $
d^-_0\neq\pm2.
$
The analogous condition is imposed on the negative branch when
$N<M$. If $d^-_0=\pm2$, the corresponding Gaussian formal invariance
equation is singular and the branch must be analysed separately.

For Theorem~\ref{thm-a-pm}, one first forms, on each branch, the coefficient
sequence $c$ of $
N^*(s,t)=N(M(s,t),t)
$
and then applies the appropriate case of the 
Theorem~\ref{thm-main-Arch} to the pair $(a,c)$. If their first discrepancy occurs
at an index greater than one, the recursion for $z(a,c)>1$ must be
used. An endpoint first coefficient $1$ or $-1$ does not permit the
higher coefficients of the branch to be discarded. 
The only singular linear Archimedean branch is the one for which the
corresponding pair of first coefficients is $(-1,1)$.
\end{rem}

\begin{rem}\label{rem-B7B8}
The exceptional pair $B_7,B_8$ can be analysed directly. Their definitions
give
\[
B_7(a,b)+B_8(a,b)=a+b.
\]
Hence the Gaussian iteration preserves the sum. If
\[
q_n=\frac{|M_n-N_n|}{M_n+N_n},
\]
then a direct calculation gives
\[
q_{n+1}=q_n\frac{|1-3q_n|}{1+q_n}.
\]
Thus $q_n\to0$, and the preserved sum shows that
\[
B_7\og B_8 \,(a,b)=\frac{a+b}{2}=A(a,b),
\qquad
B_7\og B_8 \,(x-t,x+t)=x.
\]
Except when the iteration terminates after finitely many steps, the local
relation $q_{n+1}=q_n-4q_n^2+\O(q_n^3)$ yields
$q_n\sim(4n)^{-1}$.

The Archimedean procedure is order-dependent. For $0<M_n<N_n$, put
$r_n=M_n/N_n$ and $e_n=1-r_n$. For $B_7\oa B_8$ one obtains
\[
e_{n+1}=\frac{e_n-e_n^2}{1+e_n-e_n^2}
=e_n-2e_n^2+\O(e_n^3),
\]
and therefore, the convergence is sublinear, with $e_n\sim(2n)^{-1}$. In the reversed compound $B_8\oa B_7$, the corresponding calculation
gives
\[
e_{n+1}=\frac{e_n^4}{1-e_n^2+e_n^4}=\O(e_n^4),
\]
so the local convergence is quartic. These relations also prove convergence
in both orders. At the branchwise formal level, the Gaussian pair has
$|a_1-b_1|=2$ and therefore lies on the singular boundary described in
Section \ref{section-formal}; the direct preserved identity determines the actual compound.
For the Archimedean procedure, the order $B_7\oa B_8$ has $q_a=1$ and gives
the slow rate above, whereas the reversed order has $q_a=0$ and quartic local
convergence. This example shows directly why the formal singularity and the
functional convergence question have to be treated separately.
\end{rem}

\section{Application to the compounds of the mean $M_{\alpha,r}$}

We now give a more substantial application of the branchwise expansions developed in the previous section.

In \cite{MihRais-2026}, the authors studied asymptotic properties of a new two-parameter class of means introduced in \cite{RaiRez-2019}, which is, for $a,b>0$, $a\neq b$, $r>0$, and $|\alpha|\le1$, defined by
\[
M_{\a,r}(a,b)=\frac{(r+\a)|b-a|}{(1+r|\log b-\log a|)^{1+\frac{\a}{r}}-1},
\]
and which includes some known classical means (e.g.\ the logarithmic mean), as well as other means that appear to be new. Direct expansion of the defining expression corrects the sign of the linear term in the expansion stated in \cite[Proposition 23]{MihRais-2026}:
\be\label{asymexp-Malphar}\begin{aligned}
M_{\alpha,r}(x-t,x+t) = & x-\alpha|t|+\tfrac{1}{3}(\alpha^2+2r\alpha-1)t^2x^{-1}-\tfrac{1}{3}r\alpha(2r+\alpha)|t|^3x^{-2}\\
& -\tfrac{1}{45}\left(\alpha(\alpha+2r)\left(\alpha^2-18r^2+2\alpha r-5\right)+4\right)t^4x^{-3}+\ldots 
\end{aligned}\ee

We finally apply Theorems \ref{thm-g-pm} and \ref{thm-a-pm} to the means $M_{\a,r}$ and
$M_{\b,s}$, as in the weighted-power-mean case in Section \ref{section-weightedPower}. We assume that
the relevant compounds exist and have expansions of type
\eqref{asymexp-M-abs-t}. For the Gaussian compound, assume additionally that
$\alpha\ne\beta$ and that the two means are comparable. 
The first terms of the Gaussian compound are:
\be\label{expansion-MarOGMbs}
\begin{aligned}
	M_{\a,r}&\og M_{\b,s} \,(x-t,x+t)
	={}x-\frac{\alpha+\beta}{2-|\alpha-\beta|}|t|+D t^2x^{-1}+\O(x^{-2}),\\
	D&=\frac{4\Bigl[(1+\beta\operatorname{sign}(\alpha-\beta))
	(\alpha^2+2r\alpha-1)
 {}+(1-\alpha\operatorname{sign}(\alpha-\beta))
	(\beta^2+2s\beta-1)\Bigr]}
 {3(2+|\alpha-\beta|)(2-|\alpha-\beta|)^2}.
\end{aligned}
\ee
For $0<|\alpha|<1$, from this expansion we obtain the following specific expansion:
\begin{align*}
	M_{\alpha,r}&\otimes_g M_{-\alpha,r} \,(x-t,x+t)
	= x-\frac13 t^{2}x^{-1}+\frac{|\alpha|^{2}r}{9(1-|\alpha|^{3})}\,|t|^{3}x^{-2} \\
	&\qquad +\left[-\frac{4}{45}+\frac{2|\alpha|^{2}r^{2}\bigl(4|\alpha|^{3}-19\bigr)}
	  {135\,(1-|\alpha|^{3})(1-|\alpha|^{4})}\right]\,t^{4}x^{-3}
	   +\mathcal{O}(x^{-4}).
\end{align*}
Coefficient comparison shows that $M_{\alpha,r}\otimes_g M_{-\alpha,r}$ and the logarithmic mean have  matching expansions through the term of order $x^{-1}$.

The case $\alpha=\beta$  is not covered by the formula \eqref{expansion-MarOGMbs}
above, since then $\overline{a}_1=\overline{b}_1=-\alpha$ and the first branch of
Theorem \ref{thm-main-Gauss} no longer applies. The first discrepancy between the
two coefficient sequences, with $r\neq s$, occurs at the second index,
\[
	\overline{a}_2-\overline{b}_2=\tfrac{2}{3}\alpha(r-s),
\]
so that $z(\overline{a},\overline{b})=2$ whenever $\alpha\neq0$ and $r\neq s$, and the
same holds for the negative branch. Applying \eqref{thm-gausscoeff-z>1} branchwise
gives $\underline{\gamma}_n=(-1)^n\overline{\gamma}_n$, and the compound may therefore
again be written in terms of $|t|$:
\begin{align*}
	M_{\alpha,r}\otimes_g M_{\alpha,s} \,(x-t,x+t)
	&= x-\alpha|t|+\gamma_2\,t^2x^{-1}+\gamma_3\,|t|^3x^{-2}+\mathcal{O}(x^{-3}),\\[2pt]
	\gamma_2&=\tfrac{1}{3}\bigl(\alpha^2-1+2q\alpha\bigr),\\[2pt]
	\gamma_3&=-\tfrac{1}{3}q\alpha(2q+\alpha)
	+\tfrac{1}{6}\alpha(\alpha^2-1)(r-s)^2,
\end{align*}
where
\[
	q=\tfrac{1}{2}\bigl(r+s-|\alpha|\,|r-s|\bigr).
\]

It is natural to ask whether this compound remains inside the same class. Since
$\gamma_1=-\alpha$, the first parameter is forced to be $\alpha$, and matching the
coefficient of $t^2x^{-1}$ against \eqref{asymexp-Malphar} determines the second
parameter uniquely as the above value of $q$.
Consequently, for $0<|\alpha|<1$ and $r\neq s$ the two expansions differ by
\[
	M_{\alpha,r}\otimes_g M_{\alpha,s} \,(x-t,x+t)-M_{\alpha,q}(x-t,x+t)
	=\frac{\alpha(\alpha^{2}-1)(r-s)^{2}}{6}\,|t|^{3}x^{-2}+\mathcal{O}(x^{-3}).
\]
The compound therefore does not belong to the class. Nevertheless,
$M_{\alpha,q}$ is the unique member of the family that matches the
maximal number of initial formal coefficients of the compound. Indeed,
the coefficients of $|t|$ and $t^2x^{-1}$ determine both parameters
uniquely, and no other choice of parameters can reproduce more initial
coefficients.

For $\alpha=\pm1$ the situation improves by one order. Here $q$ degenerates to a
single parameter of one of the two initial means, namely $q=\min(r,s)$, the
coefficient of $|t|^{3}x^{-2}$ matches as well, and
\[
	M_{\pm1,r}\otimes_g M_{\pm1,s} \,(x-t,x+t)-M_{\pm1,q}(x-t,x+t)
	=\pm\frac{2q\,(r-s)^{2}}{27}\,t^{4}x^{-3}+\mathcal{O}(x^{-4}),
\]
so that the two expansions agree up to and including the term of order $x^{-2}$.
Even in this extremal case the agreement remains asymptotic only: direct iteration
shows that the compound and $M_{\pm1,q}$ are different means.

For $q=0$, i.e.\ $\min(r,s)=0$, both sides degenerate: one of the two initial means is
$M_{\pm1,0}$, and the comparison mean $M_{\pm1,q}=M_{\pm1,0}$ is the same
object. It is easy to see that
\be\label{formula-Malphar-minmax}
	M_{1,0}(s,t)=\min(s,t),\quad  M_{-1,0}(s,t)=\max(s,t),
\ee
and since $\min\bigl(M_{1,r}(s,t),\min(s,t)\bigr)=\min(s,t)$ and
$\max\bigl(M_{-1,r}(s,t),\max(s,t)\bigr)=\max(s,t)$, the Gaussian iteration is
stationary already after the first step. Hence
\[\begin{split}
	M_{1,r}\og M_{1,0} \,(s,t)&=M_{1,0} (s,t)=\min(s,t),\\
	M_{-1,r}\og M_{-1,0} \,(s,t)&=M_{-1,0}(s,t)=\max(s,t),
\end{split}\]
so in this degenerate case the two means coincide as functions, and not merely
to the order of the displayed expansions.

The Archimedean compound is given by the formula:
\bes
M_{\alpha,r} \otimes_a M_{\beta,s} \,(x-t,x+t) = x +
\frac{1-\operatorname{sign}(t)\,(3\alpha+\beta)-\alpha\beta}
     {3-\operatorname{sign}(t)\,(\alpha-\beta)+\alpha\beta}\,t
+\mathcal{O}(x^{-1}), \qquad t\neq0.
\ees

The displayed formula is valid on every nonsingular branch. In
particular, it does not apply to the positive branch for
$(\alpha,\beta)=(1,-1)$ or to the negative branch for
$(\alpha,\beta)=(-1,1)$. Consequently, the coefficient comparison above does not yield a
two-sided conclusion for these parameter pairs.

The limiting case $r\to0$ of this family is defined by
\[
	M_{\alpha,0}(a,b)=\frac{\alpha\,|b-a|}{\exp\bigl(\alpha|\ln b-\ln a|\bigr)-1},
	\qquad M_{0,0}=L .
\]
Lemma 2.2 in \cite{ElVu-2014-04} is applicable here, exactly as it was applied there
to the generalized logarithmic mean at $r=-1$ and $r=0$. Indeed, for fixed $t$ the
function $z\mapsto M_{\alpha,r}(z-t,z+t)$ is analytic on $\{|z|>R\}$ and its
asymptotic expansion is the Laurent expansion of that analytic continuation; the
coefficients in \eqref{asymexp-Malphar} are polynomials in $r$, so they converge as
$r\to0$; and $M_{\alpha,r}\to M_{\alpha,0}$ pointwise, since
$(1+r\ell)^{1+\alpha/r}\to e^{\alpha\ell}$ with $\ell=|\ln b-\ln a|$. All four
hypotheses of the lemma are therefore satisfied, and the coefficients in the
asymptotic expansion of $M_{\alpha,0}$ are obtained by letting $r\to0$ in
\eqref{asymexp-Malphar}; the same applies to $M_{\beta,0}$ as $s\to0$.

For $\alpha=\pm1$ this limiting formula collapses. Based on \eqref{formula-Malphar-minmax},
on the branch $t>0$ their coefficient sequences are $(1,-1,0,0,\dots)$ and
$(1,1,0,0,\dots)$; on the branch $t<0$ the two sequences are interchanged, which is
precisely the situation described in Remark \ref{rem-branch}, where an extreme first
branch coefficient does not correspond to a two-sided expansion of type
\eqref{asymexp-M}.

The endpoint means must be handled according to the order of the Archimedean
procedure. If $M$ is a continuous strict mean, the compound of $M$ with the minimum equals the smaller of  $M(s,t)$ and $t$, while the compound of the maximum with $M$ is the maximum. Thus, for every admissible $\alpha$, $\beta$ and every
$r,s>0$,
\[
M_{\alpha,r}\oa M_{1,0} \,(s,t)=\min\bigl(M_{\alpha,r}(s,t),t\bigr),
\qquad
M_{-1,0}\oa M_{\beta,s} \,(s,t)=M_{-1,0} (s,t).
\]

Within this family, coefficient comparison did not produce general analogues of
Corollary \ref{power-cor}, although some configurations come very close: the
candidate parameters are determined by the first two non-trivial coefficients, and
the resulting mean then reproduces three or four coefficients of the compound, so it
provides the best approximation to it available within the family. A general result
of the type of Corollary \ref{power-cor} is nevertheless out of reach here. While
matching coefficients do not by themselves establish an identity of means, a single
mismatching coefficient does establish that the means are different, and this is
what occurs in each configuration considered. The exact endpoint formulas obtained
above involve $\min$ or $\max$. This illustrates
the constructive use of the algorithms announced in Section \ref{section-weightedPower}:
they are effective both for producing candidates and for ruling them out.

\section*{\normalsize Acknowledgements.}
This work was supported by the Croatian Science Foundation
(HRZZ) under the project UIP-2025-02-8956.


\end{document}